\documentclass[11pt,leqno]{article}
\usepackage[margin=1in]{geometry} 
\usepackage{amssymb,amsfonts,amsmath,bbm,mathrsfs,stmaryrd,mathtools}
\usepackage{xcolor}
\usepackage{url}

\usepackage{graphicx}

\usepackage{accents}

\usepackage{extarrows}

\usepackage[shortlabels]{enumitem}
\usepackage{tensor}

\usepackage{xr}
\usepackage[T1]{fontenc}
\usepackage[utf8]{inputenc}

\usepackage[colorlinks,
linkcolor=black!75!red,
citecolor=blue,
pdftitle={},
pdfproducer={pdfLaTeX},
pdfpagemode=None,
bookmarksopen=true,
bookmarksnumbered=true,
backref=page]{hyperref}

\usepackage{tikz}
\usetikzlibrary{arrows,calc,decorations.pathreplacing,decorations.markings,decorations.shapes,intersections,shapes.geometric,through,fit,shapes.symbols,positioning,decorations.pathmorphing}

\makeatletter
\newlength\zig@L
\newlength\zig@La
\newlength\zig@Lb

\newcommand{\xzigrightarrow}[2][]{%
  \mathrel{%
    \settowidth{\zig@La}{$\scriptstyle #2$}%
    \settowidth{\zig@Lb}{$\scriptstyle #1$}%
    \zig@L=\zig@La\relax
    \ifdim\zig@Lb>\zig@L \zig@L=\zig@Lb\fi
    \advance\zig@L by 2.2em\relax
    \tikz[baseline=-0.65ex]{%
      \draw[->,
            line cap=round,
            decorate,
            decoration={zigzag,segment length=4pt,amplitude=1.1pt}]%
        (0,0) -- (\zig@L,0)
        node[midway,above=2pt] {$\scriptstyle #2$}%
        \if\relax\detokenize{#1}\relax\else
          node[midway,below=2pt] {$\scriptstyle #1$}%
        \fi
      ;
    }%
  }%
}
\makeatother

\makeatletter
\newcommand{\squigjoin}{1mu} 

\def\sqleft@{\sim}                    
\def\sqmid@{\sim\mkern-\squigjoin}    

\def\rightsquigarrowfill@{%
  \arrowfill@{\sqleft@}{\sqmid@}{\mkern-4mu\succ}%
}

\newcommand{\xrightsquigarrow}[2][]{%
  \ext@arrow 0359\rightsquigarrowfill@{#1}{#2}%
}
\makeatother

\makeatletter
\newcommand*\circled[1]{\tikz[baseline=(char.base)]{
    \node[shape=circle, draw, inner sep=0pt, 
    minimum height={\f@size},] (char) {\vphantom{WAH1g}#1};}}
\makeatother

\makeatletter
\DeclareRobustCommand\widecheck[1]{{\mathpalette\@widecheck{#1}}}
\def\@widecheck#1#2{%
    \setbox\z@\hbox{\m@th$#1#2$}%
    \setbox\tw@\hbox{\m@th$#1%
       \widehat{%
          \vrule\@width\z@\@height\ht\z@
          \vrule\@height\z@\@width\wd\z@}$}%
    \dp\tw@-\ht\z@
    \@tempdima\ht\z@ \advance\@tempdima2\ht\tw@ \divide\@tempdima\thr@@
    \setbox\tw@\hbox{%
       \raise\@tempdima\hbox{\scalebox{1}[-1]{\lower\@tempdima\box
\tw@}}}%
    {\ooalign{\box\tw@ \cr \box\z@}}}
\makeatother

\usepackage{braket}

\usepackage[amsmath,thmmarks,hyperref]{ntheorem}
\usepackage{cleveref}

\newcommand\nthalias[1]{\AddToHook{env/#1/begin}{\crefalias{lemma}{#1}}}

\nthalias{definition}
\nthalias{example}
\nthalias{examples}
\nthalias{remark}
\nthalias{remarks}
\nthalias{convention}
\nthalias{notation}
\nthalias{construction}
\nthalias{sketch}
\nthalias{theoremN}
\nthalias{propositionN}
\nthalias{corollaryN}
\nthalias{lemma}
\nthalias{proposition}
\nthalias{corollary}
\nthalias{theorem}
\nthalias{conjecture}
\nthalias{question}
\nthalias{assumption}

\creflabelformat{enumi}{#2#1#3}

\crefname{section}{Section}{Sections}
\crefformat{section}{#2Section~#1#3} 
\Crefformat{section}{#2Section~#1#3} 

\crefname{subsection}{\S}{\S\S}
\AtBeginDocument{%
  \crefformat{subsection}{#2\S#1#3}%
  \Crefformat{subsection}{#2\S#1#3}%
}

\crefname{subsubsection}{\S}{\S\S}
\AtBeginDocument{%
  \crefformat{subsubsection}{#2\S#1#3}%
  \Crefformat{subsubsection}{#2\S#1#3}%
}

\theoremstyle{plain}

\newtheorem{lemma}{Lemma}[section]

\newtheorem{corollary}[lemma]{Corollary}
\newtheorem{theorem}[lemma]{Theorem}

\theoremstyle{plain}
\theoremnumbering{Alph}
\newtheorem{theoremN}{Theorem}

\theoremstyle{plain}
\theorembodyfont{\upshape}
\theoremsymbol{\ensuremath{\blacklozenge}}

\newtheorem{example}[lemma]{Example}

\newtheorem{remark}[lemma]{Remark}
\newtheorem{remarks}[lemma]{Remarks}

\crefname{definition}{definition}{definitions}
\crefformat{definition}{#2definition~#1#3} 
\Crefformat{definition}{#2Definition~#1#3} 

\crefname{ex}{example}{examples}
\crefformat{example}{#2example~#1#3} 
\Crefformat{example}{#2Example~#1#3} 

\crefname{exs}{example}{examples}
\crefformat{examples}{#2example~#1#3} 
\Crefformat{examples}{#2Example~#1#3} 

\crefname{remark}{remark}{remarks}
\crefformat{remark}{#2remark~#1#3} 
\Crefformat{remark}{#2Remark~#1#3} 

\crefname{remarks}{remark}{remarks}
\crefformat{remarks}{#2remark~#1#3} 
\Crefformat{remarks}{#2Remark~#1#3} 

\crefname{convention}{convention}{conventions}
\crefformat{convention}{#2convention~#1#3} 
\Crefformat{convention}{#2Convention~#1#3} 

\crefname{notation}{notation}{notations}
\crefformat{notation}{#2notation~#1#3} 
\Crefformat{notation}{#2Notation~#1#3} 

\crefname{table}{table}{tables}
\crefformat{table}{#2table~#1#3} 
\Crefformat{table}{#2Table~#1#3}

\crefname{lemma}{lemma}{lemmas}
\crefformat{lemma}{#2lemma~#1#3} 
\Crefformat{lemma}{#2Lemma~#1#3} 

\crefname{proposition}{proposition}{propositions}
\crefformat{proposition}{#2proposition~#1#3} 
\Crefformat{proposition}{#2Proposition~#1#3} 

\crefname{propositionN}{proposition}{propositions}
\crefformat{propositionN}{#2proposition~#1#3} 
\Crefformat{propositionN}{#2Proposition~#1#3} 

\crefname{corollary}{corollary}{corollaries}
\crefformat{corollary}{#2corollary~#1#3} 
\Crefformat{corollary}{#2Corollary~#1#3} 

\crefname{corollaryN}{corollary}{corollaries}
\crefformat{corollaryN}{#2corollary~#1#3} 
\Crefformat{corollaryN}{#2Corollary~#1#3} 

\crefname{theorem}{theorem}{theorems}
\crefformat{theorem}{#2theorem~#1#3} 
\Crefformat{theorem}{#2Theorem~#1#3} 

\crefname{theoremN}{theorem}{theorems}
\crefformat{theoremN}{#2theorem~#1#3} 
\Crefformat{theoremN}{#2Theorem~#1#3} 

\crefname{enumi}{}{}
\crefformat{enumi}{#2#1#3}
\Crefformat{enumi}{#2#1#3}

\crefname{assumption}{assumption}{Assumptions}
\crefformat{assumption}{#2assumption~#1#3} 
\Crefformat{assumption}{#2Assumption~#1#3} 

\crefname{construction}{construction}{Constructions}
\crefformat{construction}{#2construction~#1#3} 
\Crefformat{construction}{#2Construction~#1#3} 

\crefname{sketch}{sketch}{Sketches}
\crefformat{sketch}{#2sketch~#1#3} 
\Crefformat{sketch}{#2Sketch~#1#3} 

\crefname{question}{question}{Questions}
\crefformat{question}{#2question~#1#3} 
\Crefformat{question}{#2Question~#1#3} 

\crefname{equation}{}{}
\crefformat{equation}{(#2#1#3)} 
\Crefformat{equation}{(#2#1#3)}

\numberwithin{equation}{section}

\theoremstyle{nonumberplain}
\theoremsymbol{\ensuremath{\blacksquare}}

\newtheorem{proof}{Proof}

\newcommand\bG{{\mathbb G}}

\newcommand\bQ{{\mathbb Q}}
\newcommand\bR{{\mathbb R}}
\newcommand\bS{{\mathbb S}}
\newcommand\bT{{\mathbb T}}

\newcommand\bZ{{\mathbb Z}}

\newcommand\cB{{\mathcal B}}

\newcommand\wh{\widehat}

\DeclareMathOperator{\Ad}{Ad}

\DeclareMathOperator{\Aut}{\mathrm{Aut}}

\newcommand{\comment}[1]{}

\title{Ergodic-transformation centralizers and essentially non-compact graphing symmetry}
\author{Alexandru Chirvasitu}

\begin{document}

\date{}

\newcommand{\Addresses}{{
  \bigskip
  \footnotesize

  \textsc{Department of Mathematics, University at Buffalo}
  \par\nopagebreak
  \textsc{Buffalo, NY 14260-2900, USA}  
  \par\nopagebreak
  \textit{E-mail address}: \texttt{achirvas@buffalo.edu}

}}

\maketitle

\begin{abstract}
  We prove that for every ergodic transformation $T$ on an infinite standard probability space both the automorphism group (i.e. centralizer) $\mathrm{Aut}(T)$ and its reversing automorphism group are realizable as symmetry groups of graphings. This is an analogue of Sabidussi's realization of arbitrary graph-automorphism groups, and provides numerous examples of graphing automorphism groups carrying no compatible compact topology, answering a question of Lovasz'. Another consequence of discussion and ensuing constructions is the existence of large mutually locally-globally equivalent graphing families with highly variable symmetry. 
\end{abstract}

\noindent \emph{Key words:
  Borel graph;
  automorphism group;
  centralizer;
  ergodic;
  graphing;
  hyperfinite;
  local-global equivalence;
  standard measure space
}

\vspace{.5cm}

\noindent{MSC 2020: 20B27; 22F50; 28A60; 37A25; 28A05; 22C05; 05C63; 05C76

  
}


\section*{Introduction}

We remind the reader that a \emph{graphing} \cite[Definition 3.1]{hls_lim-grph} $\Gamma=(V,\cB,E,\mu)$ consists of
\begin{itemize}[wide]
\item a \emph{standard probability space} $(V,\cB,\mu)$ (these are the \emph{Lebesgue (probability) spaces} of \cite[\S 1.4 Definition 4.5]{pet_erg_1989}: complete probability spaces isomorphic to $[0,1]$ with its Lebesgue-measure structure together with countably many point masses);
\item together with a Borel set of edges $E\in \cB\times \cB$ constituting a loop-less graph structure on $(V,E)$ of bounded degree;
\item such that
  \begin{equation}\label{eq:edge.meas}
    \forall\left(A,B\in \cB\right)
    \left(
      \int_A\left(
        \sharp(x\to B):=\sharp\left\{\text{edges $x\to B$}\right\}
      \right)\mathrm{d}\mu(x)
      =
      \int_B\sharp\left(y\to A\right)\mathrm{d}\mu(y)\right).
  \end{equation}  
\end{itemize}
The common quantity in the equality just stipulation then defines the \emph{edge measure}
\begin{equation*}
  \eta(A\times B)
  =
  \eta_{\Gamma}(A\times B)
  :=
  \int_A \sharp(x\to B)\mathrm{d}\mu(x).
\end{equation*}

The present note is motivated in part by questions raised in \cite[\S 5]{MR4110364}, following up on the procedure, introduced there, of \emph{compactifying} a graphing. Said questions point towards possible ramifications pertaining to graphing automorphisms:
\begin{enumerate}[(a),wide]
\item\label{item:intro.which.aut} Will some sensible notion of $\Aut(\Gamma)$ be preserved by that compactification operation?

\item\label{item:intro.aut.cpct} And, taking a cue from \cite[Theorem 10]{MR3272377}, will that automorphism group be compact in an appropriate topology?
\end{enumerate}
As far as \Cref{item:intro.which.aut} goes, \cite[Definition 5.1]{MR5115419} proposes what appears to be the natural choice: automorphisms of the measure algebra $(\cB,\mu)$, respecting the edge measure in the guessable sense; or: precisely what it means to respect both the measure- and graph-theoretic structure that constitute a graphing to begin with. Given that compactification realizes $\Gamma\subseteq \wh{\Gamma}$ \emph{fully} \cite[\S 2]{MR4110364} and full embeddings are measure-theoretically undetectable, the proposed interpretation of graphing symmetry is indeed blind to the process. In short: the choice of $\Aut(\Gamma)$ answers \Cref{item:intro.which.aut} affirmatively.

\Cref{item:intro.aut.cpct}, on the other hand, will fare differently: $\Aut(\Gamma)$'s definition affords sufficient scope to ensure examples exist that are fundamentally non-compact, in the sense that no (Hausdorff) compact topology is compatible with the group structure. Amalgamating portions of \Cref{th:erg.cyc.grphng,th:all.aut.t}, a sample statement reads

\begin{theoremN}\label{thn:aut.autpm}
  For every ergodic transformation $T\circlearrowright (X,\mu)$ on an infinite standard probability space both the centralizer $\Aut(T)$ and the extended centralizer
  \begin{equation*}
    \Aut^{\pm}(T)
    :=
    \left\{S\circlearrowright (X,\mu)\ :\ STS^{-1}=T^{\pm 1}\right\}
  \end{equation*}
  are realizable as $\Aut(\Gamma)$ for ergodic graphings $\Gamma$.
  
  In particular, it is possible for $\Aut(\Gamma)$ to carry no compatible compact-group structure (e.g. when it is countably infinite, which it can be). 
\end{theoremN}
The discussion thus also ramifies in two further directions:
\begin{itemize}[wide]
\item On the one hand, building on the cursory observation \cite[Example 5.6]{MR5115419} that \emph{locally equivalent} \cite[Definition 3.3]{hls_lim-grph} graphings can exhibit variable symmetry, we will see many instances of stronger equivalence (\emph{local-global}) with the same symmetry-non-preservation feature.

\item Secondly, \Cref{thn:aut.autpm} is suggestive of prescribed-symmetry considerations, in the spirit of Sabidussi's realization \cite{sabid-props,sab-inf} of arbitrary graph-automorphism groups and the many cognates now available \cite{pt_comb-rep_1980} in the literature.
\end{itemize}


\section{Ergodic graphings and symmetries non-amenable to compactification}\label{se:ncpct}

Per \cite[Definition 5.2]{MR5115419}, an automorphism of a graphing $\Gamma=(V,\cB,E,\mu)$ is one of the \emph{measure algebra} \cite[\S 1.4C]{pet_erg_1989} (associated to) $(\cB,\mu)$ preserving the edge measure $\eta_{\Gamma}$. 

\begin{remarks}\label{res:sp.mor.2.alg.mor}
  \begin{enumerate}[(1),wide]
  \item The broader but compatible notion of \emph{iso}morphism is easily guessable: one of measure algebras, whose Cartesian square pushes one edge measure forward onto the other.  
    
  \item In the present standard measure-theoretic setting measure-algebra automorphisms are induced by measure-\emph{space} automorphisms (identified modulo measure 0, as contextually appropriate) by \cite[\S 1.4, Theorem 4.7]{pet_erg_1989}; cf. also \cite[Remark 2.3]{MR5115419}.
  \end{enumerate}  
\end{remarks}

\cite[Definition 3.3]{hls_lim-grph} introduces two notions of mutual resemblance for graphings:
\begin{itemize}[wide]
\item \emph{local (or L-)equivalence} $\Gamma_1\sim \Gamma_2$ (also \cite[\S 18.5]{lvsz_lrg-net_2012}), meaning that the probability distributions of
  \begin{equation*}
    N_{\Gamma_1,r}(x_1)
    :=
    \left(\text{$r$-neighborhood of $x_1\in V_1$}\right)
    \quad\text{and}\quad
    N_{\Gamma_2,r}(x_2)
    ,\quad
    x_i\in V_i\text{ $\mu_i$-random}
  \end{equation*}
  on the space of \emph{rooted graphs} (graphs with a distinguished vertex) of radius $\le r$ agree for all $r\in \bZ_{\ge 0}$;
  
\item and \emph{local-global (or LG-)equivalence} $\Gamma_1\approx\Gamma_2$ \cite[\S 19.2.1]{lvsz_lrg-net_2012}, a refinement of the former which takes into account graph colorings. 
\end{itemize}

\cite[Example 5.6]{MR5115419} shows that local equivalence need not entail isomorphic automorphism groups. It is perhaps worth pointing out that nor does local-global equivalence.

\begin{example}\label{ex:cyc.grphng}
  Recall the \emph{cyclic graphing} $\mathbf{C}_a$ on $\left(\bS^1\cong \bR/\bZ,\mu:=\text{Lebesgue measure}\right)$ introduced in \cite[Example 18.1]{lvsz_lrg-net_2012}: two vertices $z_1,z_2\in \bR/\bZ$ are connected precisely when they are images of reals $a>0$ apart. \Cref{th:erg.cyc.grphng} shows that for $\mathbf{C}_a$, $a\in \left(\bR\setminus \bQ\right)_{>0}$ is LG-equivalent to graphings with distinct automorphism groups (abstractly, as plain groups, disregarding any topological structure).
\end{example}

\begin{remark}\label{re:rat.vs.irrat.cyc}
  One qualitative distinction between rational and irrational cyclic graphings (rationality meaning $a\in \bQ$) is their ``tameness'' as Borel equivalence relations: rational $\mathbf{C}_a$, regarded as binary relations, are \emph{smooth} in the sense of \cite[\S 6]{km_orb_2004}, and in fact have Borel \emph{transversals} (i.e. \cite[pre Proposition 6.4]{km_orb_2004} Borel subsets in $\bS^1$ intersecting every graph connected component exactly once): a Borel section of the bundle $\bS^1\ni z\mapsto z^q\in \bS^1$ for the lowest-terms expression $a=\frac{p}{q}\in \bQ$ will provide such a transversal.

  On the other hand, \cite[Remark 4.8]{MR3754081} notes that irrational $\mathbf{C}_a$ cannot have measurable transversals. This also follows from \cite[Theorem]{MR947676}, ruling out measurable \emph{perfect matchings} in the irrational case: a measurable transversal $S\subseteq \bS^1$ would provide a measurable perfect matching by connecting every $s\in S$ with its $\mathbf{C}_a$-neighbor moving counterclockwise.
\end{remark}

\Cref{ex:cyc.grphng}, in light of passing remarks in \cite[\S 13]{hls_lim-grph}, already suffices to provide examples of LG-equivalent graphings with distinct automorphism groups: 

\begin{example}\label{ex:3.vrsn.ca}
  \cite[p.294]{hls_lim-grph} distinguishes between $\mathbf{C}_a$ and the following two variants:
  \begin{enumerate}[(a),wide]
  \item\label{item:ex:3.vrsn.ca:cvx} $\mathbf{C}'_a$ simply doubles $\mathbf{C}_a$ with each copy carrying total mass $\frac 12$;

  \item\label{item:ex:3.vrsn.ca:skw.prod} while $\mathbf{C}''_a$ consists of two mass-$\frac 12$ copies of the unit interval, with the edges connecting points on said two copies.
  \end{enumerate}
  These are both LG-equivalent to $\mathbf{C}_a$, as noted in the same cited discussion (which refers to all as realizations of the local-global limit of finite cyclic graphs). 
\end{example}

\begin{remark}\label{re:cvx.comb.grphing}
  \Cref{ex:3.vrsn.ca}\Cref{item:ex:3.vrsn.ca:cvx} generalizes to a graphing analogue of the familiar \cite[\S 10(F)]{kchrs_glob_2010} \emph{convex combination} of dynamical systems:
  \begin{equation*}
    \bigoplus_{i=1}^k \lambda_i \Gamma_i
    ,\quad
    \sum_i \lambda_i=1,\ \lambda_i\in [0,1]
  \end{equation*}
  for graphings $\Gamma_i=(V_i,\cB_i,E_i,\mu)$ is simply the disjoint union of the underlying graphs, with the measure spaces weighted by $\lambda_i$ respectively. 
\end{remark}

\Cref{ex:3.vrsn.ca} extends to broader classes of graphings that provide ample instances of LG-equivalence with large symmetry disagreement. Building on $\mathbf{C}_a$, denote by $\mathbf{C}_T$ the graphing on a standard probability space $(X,\mu)$ equipped with a $\mu$-preserving $\bZ$-action generated by $T$, with $T\ne T^{-1}$ almost everywhere ($\mathbf{C}_T$ might be termed the \emph{Cayley graphing} of the action, compatibly with \cite[\S 2]{MR3009109}, say). We will be especially interested in ergodic $T$. 

\begin{example}\label{ex:cpct.gp}
  One particular family of interest, recovering $\mathbb{C}_a$ when $\bG:=\bS^1$, would be that of graphings $\bG_a$ on compact groups $\bG$ connecting $g\in \bG$ to $ag$ for non-involutive $a\in \bG$. $\bG$ will always be assumed equipped with its Haar probability measure $\mu=\mu_{\bG}$ and, in the context of graphings, also metrizable so as to ensure that $(\bG,\mu)$ is standard.

  The translation $a\cdot$ is ergodic precisely \cite[Example 3.40.3]{gls_erg-join_2003} when $a$ generates $\bG$ topologically (so that $\bG$ is abelian).
\end{example}

\begin{theorem}\label{th:erg.cyc.grphng}
  \begin{enumerate}[(1),wide]
  \item\label{item:th:erg.cyc.grphng:gen.erg} All convex combinations
    \begin{equation*}
      \mathbf{C}_T^{\oplus \left(\lambda_i\right)_i}
      :=
      \bigoplus_{i=1}^k \lambda_i \mathbf{C}_T
      ,\quad
      \sum_i \lambda_i=1,\ \lambda_i>0
    \end{equation*}
    for ergodic $T\circlearrowright (X,\mu)$ on infinite standard probability spaces LG-equivalent, and
    \begin{equation}\label{eq:det.erg}
      \begin{aligned}
        \Aut\left(
        \mathbf{C}_T^{\oplus \left(\lambda_i\right)_i}
        \right)
        &\cong
          \Aut^{\pm}(T)\wr \Aut\left(\lambda_i\right)_{1\le i\le k}\\
        \Aut^{\pm}(T)
        &:=
          \left\{S\in \Aut(\mu)\ :\ \Ad_ST=T^{\pm 1}\right\},
      \end{aligned}
    \end{equation}
    ``$\wr$'' denoting the \emph{wreath product} \cite[p.172]{rot-gp} with the permutation automorphism group $\Aut(\lambda_i,\ 1\le i\le k)$ of the tuple $(\lambda_i)$.
    
  \item\label{item:th:erg.cyc.grphng:det.gp} In particular, all convex combinations
    \begin{equation*}
      \bG_a^{\oplus \left(\lambda_i\right)_i}
      :=
      \bigoplus_{i=1}^k \lambda_i \bG_a
      ,\quad
      \sum_i \lambda_i=1,\ \lambda_i>0
    \end{equation*}
    for topological generators $a\in \bG$ in infinite metric compact groups are mutually LG-equivalent, and
    \begin{equation}\label{eq:det.gp}
      \Aut\left(
        \bG_a^{\oplus \left(\lambda_i\right)_i}
      \right)
      \cong
      \left(\bG\rtimes \bZ/2\right)\wr \Aut\left(\lambda_i\right)_{1\le i\le k}
    \end{equation}
    with $\bZ/2$ acting as a reflection.

  \item\label{item:th:erg.cyc.grphng:non.iso} The automorphism groups in \Cref{item:th:erg.cyc.grphng:det.gp} are thus non-isomorphic whenever their identity components $\bG_0^k$ are. This happens, for instance, for distinct $k$ and a fixed finite-dimensional torus $\bG$.
  \end{enumerate}
\end{theorem}
\begin{proof}
  \begin{enumerate}[label={},wide]
  \item \textbf{\Cref{item:th:erg.cyc.grphng:det.gp} $\Rightarrow$ \Cref{item:th:erg.cyc.grphng:non.iso}:} Assume \Cref{eq:det.gp}. Because $\bG_0^k\trianglelefteq \Aut$ is abelian and \emph{divisible} \cite[Corollary 8.5]{hm5} and hence annihilated by any finite-image morphism out of $\Aut$, it can be recovered as precisely the joint kernel of all morphisms from $\Aut$ into finite groups. $k$ does indeed distinguish between the abstract isomorphism classes of $\bG^k$ for a finite-dimensional torus $\bG\cong \bT^d$: for any prime $p$ the \emph{$p$-primary component} \cite[p.5]{kap} of the torsion $t\left(\bG^k\right)$ is a sum of precisely $kd$ indecomposable summands $\bZ/p^{\infty}:=$ $p$-power-order roots of unity.

  \item \textbf{\Cref{item:th:erg.cyc.grphng:gen.erg} $\Rightarrow$ \Cref{item:th:erg.cyc.grphng:det.gp}:} This is specialization to ergodic translation actions $T:=a\cdot$, noting that
    \begin{itemize}[wide]
    \item in this case $T\cong T^{-1}$ via the inversion map, so that $\Aut^{\pm}(T)=\Aut(T)\rtimes \bZ/2$ with $\bZ/2$ generated by that map;

    \item and the centralizer $\Aut(a\cdot )$ on a compact metrizable group consists precisely \cite[Proposition 3.14]{MR2186251} of the translations and is thus identifiable with $\bG$.
    \end{itemize}
    Loc. cit. does not provide the short proof for the latter claim; for completeness, observe that for an $(a\cdot)$-centralizing automorphism $\varphi$ the sets
    \begin{equation*}
      \bG_{\Theta}
      :=
      \left\{s\in \bG\ :\ s^{-1}\cdot \varphi s \in \Theta\right\}
      ,\quad
      \text{Borel }\Theta\subseteq \bG
    \end{equation*}
    are invariant under the ergodic transformation $a\cdot$, so are all either null or full with respect to the Haar measure.    

  \item \textbf{\Cref{item:th:erg.cyc.grphng:gen.erg}:} The mutual LG-equivalence claim follows via \cite[Theorem 21.16]{lvsz_lrg-net_2012} from
    \begin{itemize}[wide]
    \item mutual \emph{local} equivalence, given that the probability measures attached \cite[post Definition 3.1]{hls_lim-grph} to all $\mathbf{C}_T^{\oplus(\lambda_i)}$ (regardless of $T$ or the $\lambda_i$, under the hypotheses) all coincide with that concentrated on the rooted bi-infinite path;
    \item together with \emph{hyperfiniteness}, again valid for all because only dependent \cite[Proposition 10.2]{hls_lim-grph} on the local equivalence class.
    \end{itemize}
    As to \Cref{eq:det.erg}, observe first that the individual weighted copies of $\mathbf{C}_T$ constituting $\mathbf{C}_T^{\oplus \left(\lambda_i\right)_i}$ are the components of the \emph{ergodic decomposition} \cite[\S 10(F)]{kchrs_glob_2010} of that graphing, so automorphisms must permute them while preserving the weights. This much already provides
    \begin{equation*}
      \Aut\left(
        \mathbf{C}_T^{\oplus \left(\lambda_i\right)_i}
      \right)
      \cong
      \Aut(\mathbf{C}_T)\wr \Aut\left(\lambda_i\right)_{1\le i\le k},
    \end{equation*}
    reducing the problem to $k=1$. It remains to argue that $\Aut(\mathbf{C}_T)\cong \Aut^{\pm}(T)$.

    Indeed, because the equivalence relation $E$ underlying $\mathbf{C}_T$ is ergodic, for any specific $\varphi\in \Aut(\mathbf{C}_T)$ one of the sets
    \begin{equation*}
      \left\{x\in X\ :\ \varphi\text{ fixes/reverses the orientation of the edge }x\to Tx\right\} 
    \end{equation*}
    must be either full or null and hence $\varphi$ either respects or reverses the orientation of \emph{every} orbit. That condition, though, translates precisely to $\varphi\in \Aut^{\pm}(T)$ as defined in the statement.
  \end{enumerate}
\end{proof}

\begin{remark}\label{re:rec.prod}
  \Cref{th:erg.cyc.grphng} covers the case of \Cref{ex:3.vrsn.ca}'s $\mathbf{C}_a''$: the latter can be identified with $\bG_g$ for $\bG:=\bS^1\times \left(\bZ/2:=\left\{1,\sigma\right\}\right)$ and $g:=\left(e^{2\pi i a},\sigma\right)$.
\end{remark}

This suffices to answer negatively (for the present notion of graphing automorphism) the last question posed in \cite[\S 5]{MR4110364}.

\begin{corollary}\label{cor:all.rvrs.gp.realiz}
  All \emph{reversing symmetry groups} \cite[\S 2, (2)]{MR2230652} $\Aut^{\pm}(T)$ for ergodic $T$ on infinite standard probability spaces are realizable as $\Aut(\Gamma)$ for $\Gamma\approx \mathbf{C}_{a\in \bR\setminus \bQ}$.

  In particular, there are graphings $\Gamma$ with $\Aut(\Gamma)$ not carrying any compatible (Hausdorff) compact-group topology. 
\end{corollary}
\begin{proof}
  The first statement is self-evident; for the second, recall \cite[\S 2]{MR399416}'s construction of a \emph{mixing} \cite[Definition 3.46]{gls_erg-join_2003} (so also ergodic) transformation $T$ whose centralizer consists precisely of its integer powers. This makes the reversing group $\Aut^{\pm}(T)$ countably infinite, for it contains $\Aut(T)\cong \bZ$ with index $\le 2$. Being \emph{perfect}, infinite compact Hausdorff groups cannot have sub-continuum cardinality \cite[Problem 30B]{wil_top}.
\end{proof}

\Cref{th:erg.cyc.grphng} in particular realizes all $\Aut^{\pm}(T)$ (for ergodic $T$ on infinite standard spaces) as automorphism groups of 2-regular graphings. Recall that in fact ``most'' ergodic $T$ are non-isomorphic to their respective inverses, so that $\Aut^{\pm}(T)=\Aut(T)$: \cite[Theorem 3]{MR633762} proves the set of such $T\circlearrowright [0,1]$ a dense $G_{\delta}$ set in the \emph{weak topology} of \cite[pp. 61-64]{halm_erg_1960} on $\Aut(\mu)$. We can improve on this somewhat, much in the spirit of many such prescribed-symmetry results recovering pre-specified groups as structure automorphism groups.

\begin{theorem}\label{th:all.aut.t}
  For every ergodic $T$ on an infinite standard Borel probability measure, $\Aut(T)\cong \Aut(\Gamma)$ for an ergodic graphing $\Gamma$ with vertex degrees $\le 3$. 
\end{theorem}
\begin{proof}
  Starting with Cayley graphing $\mathbf{C}_T$ based on $T\circlearrowright (X,\mu)$, we will replace edges by more rigid gadgets that ensure orientation preservation; the device is a measure-enhanced counterpart of sorts to the combinatorial \emph{arrow construction} employed \cite[\S IV.2]{pt_comb-rep_1980} in controlling how much symmetry one breaks when constructing objects with prescribed automorphism groups.

  For every edge $x\to Tx$ substitute the graph
  
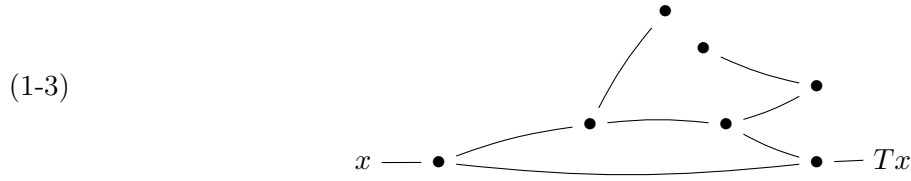
\begin{equation}\label{eq:grph.subst}
    \begin{tikzpicture}[>=stealth,auto,baseline=(current  bounding  box.center)]
      \path[anchor=base] 
      (0,0) node (l) {$\bullet$}
      +(2,.5) node (u) {$\bullet$}
      +(3.8,.5) node (r) {$\bullet$}
      +(-1,0) node (ll) {$x$}
      +(6,0) node (rr) {$Tx$}
      +(5,1) node (e) {$\bullet$}
      +(5,0) node (rd) {$\bullet$}
      +(3,2) node (ee) {$\bullet$}
      +(3.5,1.5) node (leaf2) {$\bullet$}
      ;
      \draw[-] (ll) to[bend left=0] node[pos=.5,auto] {$\scriptstyle $} (l);
      \draw[-] (l) to[bend left=6] node[pos=.5,auto] {$\scriptstyle $} (u);
      \draw[-] (u) to[bend left=6] node[pos=.5,auto] {$\scriptstyle $} (r);
      \draw[-] (l) to[bend right=6] node[pos=.5,auto,swap] {$\scriptstyle $} (rd);
      \draw[-] (rd) to[bend right=0] node[pos=.5,auto,swap] {$\scriptstyle $} (rr);
      \draw[-] (r) to[bend right=6] node[pos=.5,auto,swap] {$\scriptstyle $} (e);
      \draw[-] (r) to[bend right=6] node[pos=.5,auto,swap] {$\scriptstyle $} (rd);
      \draw[-] (u) to[bend left=6] node[pos=.5,auto,swap] {$\scriptstyle $} (ee);
      \draw[-] (e) to[bend left=6] node[pos=.5,auto,swap] {$\scriptstyle $} (leaf2);
    \end{tikzpicture}
  \end{equation}
  The underlying probability space carrying the graphing structure is the equal-weight convex combination (in the sense of \Cref{re:cvx.comb.grphing}) of the original $\left(X_0:=X,\mu\right)$ with one copy $X_{\bullet}$ carrying all instances of the respective node $\bullet$ in \Cref{eq:grph.subst}.

  The points on the various $X_{*}$ for varying $*\in \left\{*\right\}:=\{0\}\sqcup \{\bullet\in\text{\Cref{eq:grph.subst}}\}$ are distinguished by various combinatorial properties: degree sequences (i.e. common degrees of all points in $X_{*}$, the degrees of each point's neighbors and so on iteratively), distance to the nearest leaf (degree-1 vertex), etc. The requisite identification $\Aut(\Gamma)\cong \Aut(T)$ will thus be immediate once the graphing structure is in place. Note, to that end, that every edge in \Cref{eq:grph.subst} corresponds to a unique edge connecting $\circ$ to $X_{\circ'}$ for some pair $\circ\ne \circ'\in \{*\}$. Said edges, for fixed $\circ\ne \circ'$, will constitute the graph of a measured Borel isomorphism
  \begin{equation*}
    \left(X_{\circ},\mu\right)
    \xrightarrow[\quad\cong\quad]{\quad\psi_{\circ,\circ'}=\psi_{\circ',\circ}^{-1}\quad}
    \left(X_{\circ'},\mu\right);
  \end{equation*}  
  By a slight notational abuse, we also identify (the graph of) $\psi_{\circ,\circ'}$ with either $X_{\circ,\circ'}$ and hence equip it with a probability measure denoted again by $\mu$. 
  
  This describes the overall edge set $E\subset Y\times Y$ of $\Gamma$ as a Borel set. To verify the degree-symmetry condition requisite in defining a graphing structure, observe that for Borel $A_{\circ\ne \circ'}\in X_{\circ\ne \circ'}$ (respectively)
  \begin{equation*}
    \int_{A_{\circ}}\sharp\left(x\to A_{\circ'}\right)\mathrm{d}\nu(x)
    =
    \left[
      \begin{aligned}
        \mu\left(A_{\circ}\times A_{\circ'}\cap \psi_{\circ,\circ'}\right)&\text{ if $X_{\circ,\circ'}$ are connected}\\
        0&\text{ if not}
      \end{aligned}
    \right]
    =
    \int_{A_{\circ'}}\sharp\left(y\to A_{\circ}\right)\mathrm{d}\nu(y).
  \end{equation*}
  The conclusion will follow once we note ergodicity, in turn immediate given that of the original Cayley graphing $\mathbf{C}_T$ being modified and the connectedness of the edge substitutes \Cref{eq:grph.subst}.
\end{proof}



\addcontentsline{toc}{section}{References}

\def\polhk#1{\setbox0=\hbox{#1}{\ooalign{\hidewidth
  \lower1.5ex\hbox{`}\hidewidth\crcr\unhbox0}}}


\Addresses

\end{document}